\documentclass[11pt]{article}

\usepackage[top=27mm, bottom=22mm, left=34mm, right=34mm]{geometry}
\usepackage{amssymb,amsthm, amsmath}
\usepackage{setspace,hyperref}
\usepackage[usenames,dvipsnames]{color}
\definecolor{darkblue}{rgb}{0.0,0.0,0.6}
\hypersetup{colorlinks=true,breaklinks=true,
            linkcolor=darkblue,urlcolor=darkblue,
            anchorcolor=darkblue,citecolor=darkblue}
            
\allowdisplaybreaks[4]

\theoremstyle{plain}
		\newtheorem{theorem}{Theorem}
		\newtheorem{lemma}[theorem]{Lemma}
		\newtheorem{proposition}[theorem]{Proposition}

		\theoremstyle{definition}

\theoremstyle{definition}
\newtheorem*{remark}{Remark}

\theoremstyle{remark}

\usepackage[capitalise, noabbrev]{cleveref}
\crefname{theorem}{Theorem}{Theorems}
\crefname{proposition}{Proposition}{Propositions}
\crefname{corollary}{Corollary}{Corollaries}
\crefname{lemma}{Lemma}{Lemmas}
\crefname{conjecture}{Conjecture}{Conjectures}
\crefname{definition}{Definition}{Definitions}
\crefname{example}{Example}{Examples}
\crefname{remark}{Remark}{Remarks}

\newcommand{\Z}{\mathbb{Z}}

\newcommand{\col}{\operatorname{col}}

\title{Refined upper bounds on Schur-like numbers}
\author{Swaroop Hegde, Andrew Lott, Giorgis Petridis, and Nagendar Reddy Ponagandla}
\date{}

\begin{document}

\maketitle

\begin{abstract}
For positive integers $r, m$ and $N$, every \(r\)-coloring of \(\{1, \dots, N\}\) contains a monochromatic solution to $x_1+\dots+x_{m+1}=y_1+\dots+y_m$ provided that $N \ge 3^r (r!)^{1/m}$, which is qualitatively optimal when $m$ is logarithmic in $r$.
\end{abstract}

\section{Introduction}

An influential theorem of Schur states that, for any positive integer $r$ and any $r$-coloring of $\{1, \dots, N\}$, there exists a monochromatic solution to the equation $x+y=z$ provided that $N$ is large enough in terms of $r$ \cite{Schur1917}. Estimating the least such $N$ as a function of $r$ is a well-known open problem, partly because no substantial improvements have been made since Schur's original paper over 100 years ago.

A specific generalization of Schur's theorem that has been studied recently is to consider the analogous problem for the equations $x_1+\cdots+x_{m+1}=y_1+\cdots+y_m$, where $m$ is another positive integer that may depend on $r$. With this in mind, we define the \emph{Schur-like number} \(S_m(r)\) to be the least \(N\) such that every \(r\)-coloring of \([N]\) contains a monochromatic solution to 
\begin{equation}\label{eq:schurlike}
    x_1+\cdots+x_{m+1}=y_1+\cdots+y_m,
\end{equation}
where the variables in a solution are allowed to repeat. In particular, this means that Schur-like numbers are non-increasing in $m$ for a fixed $r$. 

The current state of the art regarding Schur numbers $(m=1)$, which are closely related to multicolor Ramsey numbers of triangles, is
\[
380^{r/5} \ll S_1(r) \leq (e - 1/6)r!,
\]
see \cite{Abbott-Hanson1972, ACPPRT2022, Eliahou2020, Exoo1994, Fredricksen-Sweet2000, Wan1997, Whitehead1973, XXC2002}. 

For $m=2$, Cwalina and Schoen were the first to prove that $S_2(r) = o(r!)$ as $r\to\infty$ by establishing
\[
S_2(r) \le r^{-c \log r/ \log\log r} r!
\]
for a small absolute constant $c>0$ \cite{Cwalina-Schoen2017}. Ko\'sciuszko significantly strengthened their result in \cite{Kosciuszko2025}, using a lemma of Shearer \cite{Shearer1995} to show
\[
S_2(r) \le 3^r \sqrt{(r + 1)!}.
\]

For larger values of $m$, Schur-like numbers relate to multicolor Ramsey numbers of odd cycles. Miyazaki et al. generalized Ko\'sciuszko's bound and proved in \cite{MMPZ2026}
\[
S_m(r) \le (2m + 1)^r(r!)^{1/m} + 1.
\]
Their paper is based on and sharpens a very elegant paper of Axenovich et al. on multicolor Ramsey numbers of odd cycles \cite{ACJMR2026}.

We also note that Cwalina and Schoen \cite{Cwalina-Schoen2017} and Sanders \cite{Sanders2026a} have obtained comparable bounds for the class of partition regular equations.  

The purpose of this note is to strengthen the result of Miyazaki et al. by improving the dependence on $m$ and thus offer a sharper generalization of Ko\'sciuszko's bound.

\begin{theorem} \label{thm:schurlike}
    Let $r, m, N$ be positive integers. If 
    \[
    N \ge \prod_{k=1}^r (1 + 2 k^{1/m})
    \]
    and $\{1, \dots, N\}$ is colored using $r$ colors, then there exists a monochromatic solution to the equation $x_1+\cdots+x_{m+1}=y_1+\cdots+y_m$.
\end{theorem}

In particular,
\begin{equation} \label{eq:Smr}
S_m(r) < \prod_{k=1}^r (1 + 2 k^{1/m}) \le 3^r (r!)^{1/m}.
\end{equation}
For fixed $m$, the first inequality above in \eqref{eq:Smr} takes the following form:
\[
S_m(r) \leq 2^{r + o_{r \to \infty}(r)} (r!)^{1/m}.
\]
When $m / \log r \to \infty$, the second inequality in \eqref{eq:Smr} takes the form
\[
S_m(r) \leq 3^{r+o(r)}.
\]
This exponential-in-$r$ upper bound is of the correct shape because coloring odd integers with color 1, integers congruent to 2 modulo 4 with color 2, etc., shows that, for all $r, m$, $S_m(r) \geq 2^r$ \cite{MMPZ2026}. See \cite{Fox-Kleitman2006} for related coloring constructions.

It is worth noting here that the ``toy setting'' case where one seeks the minimum $n$ for which every $r$-coloring of $(\Z/2\Z)^n \setminus \{0\}$ has a monochromatic solution to \eqref{eq:schurlike} has been solved by Sanders in the regime where $m$ is logarithmic in $r$ \cite{Sanders-personal}. Sanders has proved that the minimum dimension $n$ equals $r+1$ provided that $m$ is a sufficiently large multiple of $\log(r)$. The proof of \cite[Theorem 1.2]{ACJMR2026} gives the slightly weaker upper bound $2^n \leq 2^{r+1} r^{r/m}$. See, for example, \cite[Section 2]{Sanders2026a} and the references therein for a justification as to why studying the ``toy setting'' may be beneficial; and \cite{Abbott-Hanson1972,Heule2018,Sanders2026b} for a detailed examination of appropriately defined modular Schur numbers. 

We conclude this introduction with a quick explanation of the main idea behind the proof of Theorem~\ref{thm:schurlike}. Like many proofs of Schur's theorem, we work with a naturally defined colored graph on vertex set $\{1,\dots, N+1\}$. There, as in Miyazaki et al. \cite{MMPZ2026}, we adapt the argument of Axenovich et al. \cite{ACJMR2026}. From a quantitative perspective, their argument has three steps: an application of the pigeonhole principle to locate a suitable color followed by a telescoping argument and then by another application of the pigeonhole principle. It is the third step, the second application of the pigeonhole principle, that we complete more efficiently using the structure of the underlying graph.

\section[Proof of Theorem 1]{Proof of Theorem~\ref{thm:schurlike}}

We begin by defining the colored complete graph we use. From now on $[N] = \{1, \dots, N\}$. Let
\(
  \phi:[N]\longrightarrow[r]
\)
be an $r$-coloring and $c$ a color. We denote by  \(A_{c} \subseteq [N]\) the set of elements that are colored with color \(c\).

We form the complete, simple, undirected graph $G$ on vertex set $[N+1]$ and color its edges by
\[
  \col(xy)=\phi(|x-y|).
\]
We will need to work in induced subgraphs of this graph $G$. Every induced subgraph is also a complete graph and inherits the edge coloring from the larger graph $G$. Now fix an induced subgraph \(H\), a vertex \(v \in V(H)\), and a color \(c\). For two integers \(0\leq j\leq i\), we define \(Z_{i,j}:= Z(H,v,c,i,j)\) to be the set of endpoints of walks 
\[
  v=v_0,v_1,\ldots,v_i=x
\] 
of length $i$ in $H$, all of whose edges have color $c$, such that exactly $(i-j)$ steps satisfy $v_\ell-v_{\ell-1}\in A_c$, and exactly $j$ steps satisfy $v_\ell-v_{\ell-1}\in -A_c$. Repeated vertices and backtracking are allowed. In particular, every element of $Z_{i,j}$ can be expressed as
\begin{equation}  
 x=v+a_{1}+\dots+a_{i-j}-a'_{1}-\dots-a'_{j},
\label{eq:charge-rep}
\end{equation}
where all the \(a_\ell\) and \(a'_\ell\) lie in \(A_c\). The set \(Z_{0,0}\) consists only of the vertex \(v\).

We record two simple yet important observations. The first is recorded in many places in the literature, e.g., in \cite{Kosciuszko2025}. The second, which is even simpler, is what allows efficient pigeonholing.

\begin{lemma}\label{lem:independence}
Let \(H\) be an induced subgraph of the graph $G$ defined at the beginning of the section, let \(v \in V(H)\), let \(c\) be a color incident to \(v\), and let $m$ be a positive integer. If there are no monochromatic solutions to the Schur-like equation \eqref{eq:schurlike}, then for every \(0 \leq j \leq i\leq m\) the set \(Z_{i,j}\) is color-\(c\) independent.
\end{lemma}

\begin{proof} If $i=0$, which forces $j=0$, this is immediate since $Z_{i,j}=\{v\}$. Next assume $i \geq 1$ and consider two distinct elements \(u\) and \(w\) of \(Z_{i,j}\) with \(u<w\). Suppose that there exists a color-\(c\) edge between them. There must then exist $a \in A_c$ such that $w = u + a$. Plugging in the representations of \(u\) and \(w\) as in \eqref{eq:charge-rep} and rearranging, we obtain a solution to a Schur-like equation as in \eqref{eq:schurlike}, with \(i+1\) terms on one side and \(i\) terms on the other side. This is not possible, giving the contradiction we were after.
\end{proof}

\begin{lemma}\label{lem:branching}
Let \(H\) be an induced subgraph of the graph $G$ defined at the beginning of the section, let \(v \in V(H)\), and let \(c\) be a color incident to \(v\). For all \(0\leq j \leq i\), we have 
\[
 N_{c}^{H}(Z_{i,j}) \subseteq Z_{i+1,j} \cup Z_{i+1,j+1},
 \] 
 where
\(N_{c}^{H}(Z)\) denotes the color-\(c\) neighborhood of \(Z\) inside \(H\).
\end{lemma}

\begin{proof} 
This is immediate since a color-\(c\) edge corresponds to either adding or subtracting an element of \(A_{c}\).
\end{proof}

The next step in setting up the proof is to introduce weights on vertices. This novel ingredient is one of the key ideas in the work of Axenovich et al. \cite{ACJMR2026}. We further modify the slightly more efficient weights used by Miyazaki et al. \cite{MMPZ2026}.

The \emph{color-degree} \(d_{H}(u)\) of \(u \in V(H)\) is the number of distinct colors incident to $u$ in \(H\). We assign to \(u\) the following weight
\begin{equation} \label{eq:weights}
 w_H(u) = \prod_{k=1}^{d_H(u)} \frac{1}{(1+ 2 k^{1/m})}.
\end{equation}
 The total weight of
a subset \(U\subseteq V(H)\) is defined as
\[
 w_{H}(U)=\sum_{u\in U}w_{H}(u). 
 \]
The following upper bound on the total weight of an induced graph, a slightly stronger version of similar statements in \cite{ACJMR2026,MMPZ2026}, is the key result.

\begin{proposition} \label{prop:weight_le1}
Let $m\geq 1$ and \(H\) be an induced subgraph of the graph $G$ defined at the beginning of this section equipped with vertex weights as defined in \eqref{eq:weights}. If there are no monochromatic solutions to the Schur-like equation \eqref{eq:schurlike}, then \(w_{H}(V(H))\leq 1\).
\end{proposition} 

\begin{proof} 
We proceed by induction on \(|V(H)|\). If \(H\) has a single vertex, then the result is trivial.

Next suppose \(|V(H)|>1\) and set \(W:=w_{H}(V(H))\). To simplify notation, we suppress the dependence on $H$ and write $N_c^{H}$ as $N_c$, and $w_H$ as $w$, unless it is necessary to make a distinction. Choose a vertex \(v\) of minimum color-degree and put \(d=d_{H}(v)\). Next, by the pigeonhole principle, choose a color \(c\) that satisfies
\[
 w(N_{c}(v)) \geq \frac{W-w(v)}{d}.
 \] 
 We consider \(Z_{i,j}\) for \(0\leq i\leq m\) and \(0\leq j \leq i\). There are two complementary cases to consider.

\textbf{Case 1:} Suppose there exist \(1\leq i \leq m\) and \(0\leq j\leq i\) such that \(Z_{i,j}\neq\varnothing\) and 
\begin{equation} \label{eq:case1}
     w(N_{c}(Z_{i,j})) \leq 2 d^{1/m}w(Z_{i,j}). 
\end{equation}

We then delete \(N_{c}(Z_{i,j})\) and let \(H'\) be the subgraph of
\(H\) induced on \(V(H)\setminus N_{c}(Z_{i,j})\). By Lemma~\ref{lem:independence}, \(Z_{i,j}\) contains no color-\(c\) edges and hence survives the deletion. Crucially, every vertex in \(Z_{i,j}\) loses its color-\(c\) edges and therefore its color-degree decreases by at least 1. So, in the new induced graph $H'$ the weights of vertices in \(Z_{i,j}\) are multiplied by a factor of at least \((1+ 2 d^{1/m})\) (by their definition in \eqref{eq:weights} and the minimality of \(d\)). The weights of all other vertices do not decrease and thus,
 \begin{align*}
 w_{H'}(V(H'))  & \geq W - w_{H}(N_{c}^{H}(Z_{i,j}))+ ((1+ 2 d^{1/m})-1)w_{H}(Z_{i,j}) \\
  & \geq W  \qquad (\text{by \eqref{eq:case1}}).
 \end{align*} 

By the inductive hypothesis we have $w_{H'}(V(H')) \le 1$, and we get the desired upper bound \(W\leq 1\).

\textbf{Case 2:} Suppose that for all \( 1\leq i\leq m\) and \(0\leq j \leq i\) such that \(Z_{i,j}\neq\varnothing\), we have
\begin{equation} \label{eq:case2}
 w(N_{c}(Z_{i,j})) > 2 d^{1/m} w(Z_{i,j}).
\end{equation}

We start by setting \(Y_{0}:=Z_{0,0}=\{v\}\) and note that
\(N_{c}(v)=Z_{1,0} \cup Z_{1,1}\). So choose \(j_{1}\in \{ 0,1 \}\)
such that \[
 w(Z_{1,j_{1} })\geq \frac{1}{2} w(N_{c}(v)). \]

By the choice of \(c\) we have

\begin{equation}
 w(Z_{1,j_1})\geq \frac{W-w(v)}{2d}.
\label{eq:Z1}
\end{equation}

Set $Y_1 := Z_{1,j_1}$. We now iterate the following procedure. Suppose \(Y_i=Z_{i,j_i}\) is
nonempty for some \(1\leq i\leq m-1\). By \eqref{eq:case2}, we have

\begin{equation*}
w(N_c(Z_{i,j_i})) > 2 d^{1/m}w(Z_{i,j_i}).
\end{equation*}

Applying Lemma~\ref{lem:branching}, we also have
\begin{equation}\label{eq:branching}
    w(N_{c}(Z_{i,j_{i}})) \leq w(Z_{i+1,j_{i}})+w(Z_{i+1,j_{i}+1}).
\end{equation}

Choose \(j_{i+1}\in\{j_i,j_i+1\}\) so that
\(Y_{i+1}:=Z_{i+1,j_{i+1}}\) has at least as much weight as the other of the two
sets on the right-hand side of \eqref{eq:branching} (and in particular is nonempty). Thus, by combining
the two inequalities above, we have 
\[
 w(Z_{i+1,j_{i+1}}) > d^{1/m} w(Z_{i,j_{i}}).
 \]

Starting from level 1 and applying this \(m-1\) times, we obtain a set
\(Z_{m,j_{m}}\) such that
\begin{align*} 
w(Z_{m,j_{m}})  & >  d^{(m-1)/m}w(Z_{1,j_{1}})  \\
 & \geq \frac{W-w(v)}{2d^{1/m}},
\end{align*}
where we have used \eqref{eq:Z1}.

By Lemma~\ref{lem:independence}, there are no color-\(c\) edges inside
\(Z_{m,j_m}\), and, by definition, each vertex in \(Z_{m,j_m}\) has a color-\(c\) edge incident to it. We now pass to the induced subgraph \(H'=H[Z_{m,j_m}]\), where the color-degrees of vertices in \(Z_{m,j_m}\)
go down by at least 1 in $H'$, and hence their weights are multiplied by a factor of at least $(1+2d^{1/m})$ (by the minimality of $d$). Therefore
\begin{align*}
  w_{H'}(V(H'))  & \geq (1+2d^{1/m}) w_{H}(Z_{m,j_m})  \\
   & > \left(1 + \frac{1}{2 d^{1/m}}\right) (W-w_{H}(v)).
 \end{align*} 

 We now use two inequalities. The first is that for all vertices $v \in V(H)$
 \[
w_H(v) \leq \frac{1}{3} \dots \frac{1}{1+2 d^{1/m}} \leq \frac{1}{3 d^{1/m}}.
 \]
The second inequality, which follows from $\tfrac{1}{2}-\tfrac{1}{3} = \tfrac{1}{2} \cdot \tfrac{1}{3}$ and $d^{1/m} \ge 1$, is
\[
1+ \frac{1}{2 d^{1/m}} \geq \frac{1}{1-\frac{1}{3d^{1/m}}}.
\]
Substituting above yields
\[
w_{H'}(V(H')) \left(1-\frac{1}{3d^{1/m}}\right) > \left( W - \frac{1}{3 d^{1/m}}\right).
\]
Finally, since \(H'\) is a proper induced subgraph of \(H\), the inductive hypothesis gives \(w_{H'}(V(H')) \leq 1\) and the desired bound \(W<1\) follows, completing Case 2.

\end{proof}

Theorem~\ref{thm:schurlike} is an immediate corollary.

\begin{proof}[Proof of Theorem~\ref{thm:schurlike}]
We prove the contrapositive. Suppose that there are no monochromatic solutions to equation \eqref{eq:schurlike}. By Proposition~\ref{prop:weight_le1}, the total weight of the original graph $G$ is at most 1. Since every vertex has color-degree at most $r$, we have
\begin{align*}
    (N+1) \prod_{k=1}^r \frac{1}{1 + 2 k^{1/m}} \leq \sum_{i=1}^{N+1} w_G(i) \leq 1.
\end{align*}
\end{proof}

\begin{remark}
    More careful bookkeeping, together with the known values of small Schur numbers, yields 
    \[
    S_m(r) \leq 5 \prod_{k=3}^r (1 + 2^{1-1/m} k^{1/m}) .
    \]
\end{remark}

\subsubsection*{Acknowledgments}
The authors would like to thank Cosmin Pohoata and Tom Sanders for helpful conversations. Research on this project was supported by the Simons Foundation grant MPS-TSM-00007816. Part of the research was carried out when all four authors were visiting the HUN-REN Alfr\'ed R\'enyi Institute of Mathematics (Erd\H{o}s Center) during the Simons Workshop on Harmonic Analysis and Applications to Ramsey Theory. They thank the R\'enyi Institute for its hospitality and support. The first and third authors were supported by the HUN-REN Alfr\'ed R\'enyi Institute of Mathematics (Erd\H{o}s Center) while part of this project was carried out.

\subsubsection*{AI disclosure}
The authors used OpenAI's ChatGPT, specifically ChatGPT 5.6 Sol Ultra. The authors initially attempted to use directed graphs, which corresponds to solely using the sets $Z_{i,0}$. This choice, while allowing for efficient pigeonholing, created some problems. ChatGPT suggested using sets similar to the $Z_{i,j}$ that were eventually used. ChatGPT was also used in the preparation of the manuscript. The mathematical arguments presented are the authors' own, and the authors take full responsibility for the contents of the paper.

\phantomsection

\addcontentsline{toc}{section}{References}

\bibliographystyle{plain}


\bibliography{/Users/petridis/Dropbox/bib/all} 


\hspace{20pt} Department of Mathematics, University of Georgia, Athens, GA 30602, USA.

\end{document}